\documentclass[reqno,11pt]{article}
\usepackage{amsmath,amssymb}
\usepackage{amsthm,amscd,amsfonts,}
\usepackage{amsmath}
\usepackage{amsthm}
\usepackage{graphicx}
\usepackage{amsmath,amsthm}
\usepackage{amsmath,amstext,amsfonts}
\usepackage{subfigure}
\usepackage[T1]{fontenc}
\usepackage{authblk}

\newtheorem{theorem}{Theorem}[section]

\newtheorem{remark}{Remark}[section]

\newtheorem{example}{Example}[section]

\newcommand{\R}{\mathbb{R}}
\numberwithin{equation}{section}

\begin{document}

\title{\textbf{On Solving Some Classes of Second Order ODEs}}

\author{R. AlAhmad\thanks{rami\_thenat@yu.edu.jo}}
\author{M. Al-Jararha\thanks{mohammad.ja@yu.edu.jo}}
\affil{Department of Mathematics, Yarmouk University, Irbid, Jordan, 21163.}
 
\date{}
\maketitle

\begin{abstract}
In this paper, we  introduce some analytical techniques to solve some classes of  second order differential equations. Such classes of differential equations arise in describing some mathematical problems in Physics and Engineering. 
\end{abstract}
%============================================================================
\vspace{0.5 cm}
%=============================================================
\section{Introduction}
One of the most important applications in the calculus of variation is to maximize  (minimize) the functional

\begin{equation}\label{integralfunctional}
Q[y]=\int_{a}^{b} \left( \sqrt{p(x)} (y^\prime(x))^2 + \frac{ h(y(x))}{ \sqrt{p(x)}}\right) dx,
\end{equation}
 
where $p(x)$ is a positive and differentiable function on some open interval $(a,b)\subset \R,$ and $h(x)$ is a differentiable function. In fact, the functional 
\begin{equation}\label{integralfunctional2}
Q[y]=\int_{a}^{b} \left( \sqrt{p(x)} (y^\prime(x))^2 + \frac{ h(y(x))}{ \sqrt{p(x)}}\right) dx
\end{equation}
attains its extreme values at a function  $y(x)\in C^2(a,b)$ that  $y(x)$  satisfies the  Euler's-Lagrange differential equation \cite{Arfken, Hand, Makarets},
\begin{equation}\label{leequation}
\frac{\partial F}{\partial y}(x,y,y^\prime) - \frac{d}{dx}\left(\frac{\partial F}{\partial y^\prime}(x,y,y^\prime)\right)= 0,
\end{equation} 
where 
\begin{equation}\label{functionform}
F(x,y,y^\prime):=\sqrt{p(x)} (y^\prime(x))^2 + \frac{h(y(x))}{ \sqrt{p(x)}}.
\end{equation}
Therefore, the problem of maximizing (minimizing)  $Q[y]$ is reduced to solve the differential equation \eqref{leequation}. i.e., to solve %Substituting $F(x,y,y^\prime)$ given above  in Equation \eqref{leequation}, we get the following second order differential equation:
\begin{equation}\label{primerydiffequation}
p(x)y^{\prime\prime}(x)+\frac{1}{2}p^\prime(x)y^\prime(x)=\frac{1}{2}h^\prime(y(x)).
\end{equation}
%Therefore, by solving the above equation, we can find a function $y(x)\in C^2(a,b)$ that is  maximizing (minimizing) $Q[y].$ 
Hence, it is a matter to solve such differential equations. In the first part of this paper, we solve the following class of second order differential equation:
\begin{equation}\label{main}
p(x)y^{\prime\prime}(x)+\frac{1}{2}p^\prime(x)y^\prime(x)+f(\sqrt{p(x)}\;y^\prime(x),y(x))=0
\end{equation}
which generalizes the differential equation \eqref{primerydiffequation}. Here, we assume that $p(x)$ is a positive and differentiable function on some open interval $ (a,b)\subset \R$, and  $f(\sqrt{p(x)}\;y^\prime(x),y(x))$ is continuous function on some domain $D\subset \R^2$. In fact, equation \eqref{main} not only generalizes  \eqref{primerydiffequation}, but also it generalizes many of well known differential equation. For example,%it is important to \eqref{main} since it is a generlization of  \eqref{primerydiffequation}, and hence, it is generalize some of very well known differential equations.  For example, of such equations are 

\begin{enumerate}
\item the Chebyshev's Differential Equation ~\cite{Arfken, Zw},
 \[
 (1 -x^2)y^{\prime\prime}(x)-x\;y^\prime(x)+n^2 y(x)=0,\;\; |x|<1,
 \]
\item the Cauchy-Euler's Differential Equation ~\cite{Boyce, Kr, MURPHY}, 
\[
ax^2y^{\prime\prime}(x)+a\;x\;y^\prime(x)
+b\;y(x)=0,\;\; x>0,
\]
\item the Nonlinear Chebyshev's Equation,
\[
 (1 -x^2)y^{\prime\prime}(x)+\left(\alpha\sqrt{1-x^2}-x\right)y^\prime(x) +f(y(x)) = 0,\;\; \;\;|x|<1,
\]

\item the Hypergeometric Differential Equation  ~\cite{Anderws,Zw}, 
\[
x(1-x)y^{\prime\prime}(x) + \left[c-(a+b+1)x \right] y^\prime(x) - a\,b\;y(x) = 0, \;0<x<1, \;\text{ with } c=1/2, a=-b,
\]
and
\item  the Nonlinear Hypergeometric Differential Equation, 
\begin{equation*}%\label{hyper}
x(1-x)y^{\prime\prime}(x)+ \left(\frac{1}{2}-x+\alpha\sqrt{x(1-x)}\right)y^\prime(x)+f(y(x))=0,\; 0<x<1.
\end{equation*}
\end{enumerate}
Throughout this paper, we call the class of differential equation in \eqref{main} by Chebyshev's-type of differential equations.

In the second part of this paper, we introduce an approach to solve %the following class of nonlinear second order differential equations:
the differential equation
\begin{equation}\label{part2eq}
a_2\left(f^\prime(y)y^{\prime\prime}+(y^\prime)^2f^{\prime\prime}(y)\right)+a_1f^\prime(y)y^\prime+a_0f(y)=g(x),
\end{equation}
where $a_0,\;a_1$ and $a_2$ are constants, and $f\in C^2(a,b)$, for some  open interval $(a,b) \subset \R$. We also give an approach to solve the differential equation %the following class of nonlinear second order differential equations:
\begin{equation}\label{part2eeq}
p(x)\left(f^\prime(y)y^{\prime\prime}+(y^\prime)^2f^{\prime\prime}(y)\right)+\frac{1}{2}p^\prime(x)f^\prime(y)y^\prime+a_0f(y)=0,
\end{equation}
where $p(x)$ is a positive and differentiable function on some open interval $(a,b) \subset \R$, and  $f\in C^2(c,d)$, for some  open interval $(c,d) \subset \R$.
Throughout this paper, we call these classes of second order differential equations by $f-$type second order differential equations.

In the third part of this paper, we introduce an approach to solve %the following class of nonlinear 
the second order nonlinear differential equation
\begin{equation}
a_2(x,y,y^\prime)\left(f^\prime(y)y^{\prime\prime}+f^{\prime\prime}(y)(y^\prime)^2\right)+a_1(x,y,y^\prime)(f^\prime(y)y^\prime)+a_0(x,y,y^\prime)=0,
\end{equation}
where  $f(y)$ is an invertible function ($y=f^{-1}(z)$), and  $f\in C^2(a,b)$ where $(a,b)$ is the open interval in $\R.$ To solve this class of differential equations, we assume that
\begin{equation}\label{exact}
a_2\left(x,f^{-1}(z),\frac{z^\prime}{f^\prime\left(f^{-1}(z)\right)}\right)z^{\prime\prime}+a_1\left(x,f^{-1}(z),\frac{z^\prime}{f^\prime\left(f^{-1}(z)\right)}\right)z^\prime+a_0\left(x,f^{-1}(z),\frac{z^\prime}{f^\prime\left(f^{-1}(z)\right)}\right)=0
\end{equation}
is exact differential equation. The differential equation  \eqref{exact} is called exact  if the conditions 
\begin{equation}\label{exactcond2}
\frac{\partial a_2}{\partial z}=\frac{\partial a_1}{\partial z^\prime}, \;\; \frac{\partial a_2}{\partial x}=\frac{\partial a_0}{\partial z^\prime},\; \text{and}\; \;\frac{\partial a_1}{\partial x}=\frac{\partial a_0}{\partial z}.
\end{equation}
hold \cite{AlAhmad,Aljararha}.
In this case, the first integral of \eqref{exact}  exists and it is given by % be reduced into the following first order differential equation \cite{AlAhmad,Aljararha}:
\[
\int_{x_0}^{x}a_0(\alpha,z,z^\prime)d\alpha+\int_{z_0}^{z}a_1(x_0,\beta,z^\prime)d\beta+\int_{z^\prime_0}^{z^\prime}a_2(x_0,z_0,\gamma)d\gamma=c.
\]
Throughout this paper, we call this class of differential equations by $f-$type second order differential equations that can be transformed into exact second order differential equations. 

The layout of the paper: In the first section, we solve Chebyshev's-type of Second Order Differential Equations. In the second section, we solve the $f-$type of second order differential equations. In the third section, we solve  $f-$type second order differential equation that can be transformed into exact second order differential equations. The fourth section is devoted for the concluding remarks. 
%==================================================================================
\section{Solving Chebyshev's-type of Second Order Differential Equations}
In this section, we present an approach to solve Chebyshev's-type of second order differential equations
\begin{equation}\label{mainsec1}
p(x)y^{\prime\prime}(x)+\frac{1}{2}p^\prime(x)y^\prime(x)+f(\sqrt{p(x)}\;y^\prime(x),y(x))=0,
\end{equation}
where $p(x)$ is a positive and differentiable function on some open interval $\mathbb \mathbb (a,b) \in \R$, and  $f(\sqrt{p(x)}\;\\y^\prime(x),y(x))$ is continuous function on some domain $D\subset \R^2$.  The approach is described in the following theorem:
\begin{theorem}\label{maintheorem}
Assume that $p(x)$ be a positive and differentiable function on the open interval $\mathbb (a,b)\subset \R$. Let $x_0$ be any point in the interval $(a,b)$. Then 
\[
  t=\int_{x_0}^x\frac{d\xi}{\sqrt{p(\xi)}}
\] 
 transforms the differential equation ~\eqref{mainsec1} into the second order differential equation
\begin{equation}\label{main22}
 y^{\prime\prime}(t)+f(y(t),y^\prime(t))=0.
\end{equation} 
\end{theorem}
\begin{proof} Let $$t=\int_{x_0}^x\frac{d\xi}{\sqrt{p(\xi)}}.$$ Since,  
\[
\frac{dy}{dx}=\frac{dy}{dt}\frac{dt}{dx}.
\]
Hence,
\begin{equation}\label{main4}
\frac{dy}{dt}= \sqrt{p(x)}\;\frac{dy}{dx}.
\end{equation}
Therefore,
\begin{equation}\label{main5}
\frac{d^2y}{dt^2}
=\frac{d}{dt}\left(\sqrt{p(x)}\;\frac{dy}{dx}\right)=\frac{d}{dx}\left(\sqrt{p(x)}\;\frac{dy}{dx}\right)\frac{dx}{dt}=p(x)y^{\prime\prime}+\frac{1}{2} p^\prime(x)y^\prime.
\end{equation}
By substituting  ~\eqref{main4} and  ~\eqref{main5} in Equation ~\eqref{main}, we get the result.\quad
\end{proof}
\begin{remark}\label{R2}
The differential equation ~\eqref{main22} is independent of the variable $t$, and so, it is easy to solve by setting $\eta(t)=y^\prime(t)$. Hence, it reduces into the following first order differential equation: 
\begin{equation}\label{main3}
\eta\frac{d\eta}{dy}+f(y,\eta)=0.\
\end{equation}
%\end{remark}
%The above differential equation is easy to solve, and it can be solved by applying the elementary techniques to solve first order differential equations.
%\begin{remark} \label{R2}
In case that $f(\sqrt {p(x)}\;y^\prime,y)=f(y)$, we get %then Equation \eqref{main3} becomes 
%\begin{equation}%\label{main3}
%\eta\frac{d\eta}{dy}+f(y)=0.
%\end{equation}
%Therefore, 
\[
\eta^2(t)=-2\int^y f(\xi)d\xi+c.
\]
Hence,
\[
y^\prime(t)=\left(c-2\int ^y f(\xi)d\xi\right)^{\frac{1}{2}}
\]
where $c$ is the integration constant.
\end{remark}
Next, we present some examples to explain this approach.%solve such class of differential equations.
\begin{example}
Consider the nonlinear Chebyshev's differential equation
\begin{equation}\label{hh}
\left\{
\begin{array}{ll}
(1 -x^2)y^{\prime\prime}(x)-x y^\prime(x) +4\sqrt{1-x^2}\;y^{\prime}y(x) = 0,\\
\\
y(0)=\frac{1}{2},\;y^\prime(0)=-\frac{1}{2}\,.
\end{array}
\right.
\end{equation}
Then
\[
t=\int^x\frac{d\xi}{\sqrt{1-\xi^2}}=\arcsin(x)
\]
transforms \eqref{hh} into 
\[
y^{\prime\prime}(t)+4y^\prime(t) y(t)=0.
\]
Set $\eta(t)=y^\prime(t)$. The above equation becomes  
\[
\eta \frac{d \eta}{dy}+4\eta y=0. 
\]
%Hence,
%\[
%\eta \left(\frac{d \eta}{dy}+4 y\right)=0.
%\]
%Therefore, either $\eta(t)=y^\prime(t)=0$ or 
%$
%\frac{d \eta}{dy}+4 y=0.
%$
%Since $\eta(t)=y^\prime(t)=0$ contradicts the initial conditions, we have
%$$
%\frac{d \eta}{dy}+4 y=0.
%$$
%This implies that
%\[
% \eta+2 y^2=0.
%\]
The solution of this equation is $y(t)=\displaystyle\frac{1}{2(t+1)}.$ Therefore, $y(x)=\displaystyle\frac{1}{2(\arcsin(x)+1)}.$
\end{example}

\begin{example} Consider the initial value problem 
\begin{equation}\label{ex2eq}
\left\{
\begin{array}{ll}
x^2y^{\prime\prime}+xy^\prime- 3y^2=0, \;x>0, \\
\;y(1) = 2, \;y^\prime(1) = 4.
\end{array}
\right.
\end{equation} 
Then  
$$t=\int_{1}^x \frac{d\xi}{\xi}d\xi=\ln(x)$$ transforms the  \eqref{ex2eq} into 
\begin{equation}\label{transone}
\left\{
 \begin{array}{ll}
 y^{\prime\prime}-3y^2=0,\\
 y(0) = 2,\; y^\prime(0) = 4.
 \end{array}
\right.
\end{equation} 
%Using Remark ~\ref{R2}, the above equation is reduced to the following first order differential equation:
%\[
%\eta\frac{d\eta}{dy}-3y^2=0,
%\] 
%which can be solved by using the separating of variables technique. This implies that 
The solution of the above differential equation is
\[
y (t)= \frac{2}{(1 - t)^2}.
\]  
%By substituting $t=\ln x$, 
Hence, %we get that
\[
y(x) = \frac{2}{(1 - \ln(x))^2}\,.
\]
\end{example}
\begin{example}\label{ex3} 
Consider the  linear form of ~\eqref{main} 
\begin{equation}\label{ex1eq}
\phi(x)y^{\prime\prime}+\frac{1}{2}\phi^\prime(x)y^\prime+\lambda^2 y=0, 
\end{equation} 
where $\lambda\in \R,$ and $\phi(x)$ is a positive and differentiable function on some open interval $\mathbb (a,b) \subset \R$. By applying the transformation 
\[
  t=\int_{x_0}^x \frac{d\xi}{\sqrt{\phi(\xi)}}d\xi,
\]
equation \eqref{ex1eq} can be transformed into the following second order differential equation:
$$\frac{d^2y}{dt^2}+\lambda^2 y=0.$$ 
The solution of this different5ial equation is %equation is very simple, and its general solution  is  given by  
$$y(t)=C_1\sin(\lambda t)+C_2\cos(\lambda t).$$ 
Hence, the general solution of equation ~\eqref{ex1eq} is 
$$y(x)=C_1\sin\left(\lambda\int_{x_0}^x \frac{d\xi}{\sqrt{\phi(\xi)}}d\xi \right)+C_2\cos\left(\lambda\int_{x_0}^x \frac{d\xi}{\sqrt{\phi(\xi)}}d\xi \right).$$ %is the general solution of this differential equation.
\end{example}
\begin{example}
Consider the following second order  linear differential equation (see Eq. 239, p. 335 in \cite{MURPHY}):
\begin{equation}\label{part1ex4}
4xy^{\prime\prime}+2y^\prime+y=0.
\end{equation}
From the previous example, the general solution of this equation is given by
$$y(x)=C_1\sin\left(\int_{x_0}^x \frac{d\xi}{2\,\sqrt{\xi}}\,d\xi \right)+C_2\cos\left(\int_{x_0}^x \frac{d\xi}{2\,\sqrt{\xi}}\,d\xi \right)$$
and so, 
\[
y(x)=C_1\sin\left(\sqrt{x}\right)+C_2\cos\left(\sqrt{x}\right).
\]
\end{example}
\begin{remark} \label{R3}
Consider the  second order linear differential equation
\begin{equation}
(\phi(x))^2y^{\prime\prime}(x)+\phi(x)\phi^\prime(x)y^\prime(x)+\lambda y(x)=0, \;\;\; x \in (a,b),
\end{equation}
and assume that $\phi(x)$ is a positive and differentiable function on an open interval $(a,b)\subset \R.$ Moreover, assume that $\phi(a)=\phi(b)=0$. Define the linear differential operator 
\[
L[y]:=-\left( (\phi(x))^2y^{\prime\prime}(x)+\phi(x)\phi^\prime(x)y^\prime(x))\right)=\lambda y(x).
\]
Then the boundary value problem 
\begin{equation}\label{bvpp}
\left\{ \begin{array}{ll}
L[y]=-\left((\phi(x))^2y^{\prime\prime}(x)+\phi(x)\phi^\prime(x)y^\prime(x)\right)=\lambda y(x),& a<x<b,\\
%\alpha_1y(a)+\beta_1y^\prime(a)=0, &\\
%\alpha_2y(a)+\beta_2y^\prime(a)=0, &\\
\phi(a)=\phi(b)=0,
\end{array}
\right.
\end{equation}
satisfies the Lagrange Identity $\int_a^b \phi L[\psi]dx=\int_a^b \psi L[\phi]dx$, where $\phi$ and $\psi$ satisfy the above boundary value problem. Therefore, the operator $L[y]$ is self-adjoint. Hence, the boundary value problem \eqref{bvpp} has an orthogonal set of eigenfunctions $\{\phi_n(x)\}_{n=1}^\infty$ with corresponding eigenvalues $\{\lambda_n\}_{n=1}^\infty$. Since the above boundary value problem is a special case of \eqref{main}. Then, by using the approach described in Theorem \ref{maintheorem}, it is easy to find its  orthogonal set of eigenfunctions.   
\end{remark}

By using the same approach described in Theorem \ref{maintheorem}. We can solve the following class of second  order linear differential equations:
\begin{equation}\label{mainlin1}
\left[P(x)\right]^2y^{\prime\prime}(x)+P(x)\left[\alpha+ P^\prime(x) \right]y^\prime(x)+\beta y(x)=0, 
\end{equation} 
where $P(x)>0$, $P(x)\in C^1(a,b)$,  and $\alpha$ and $\beta$ are constants.
In fact, the transformation 
\begin{equation} \label{transformation}
t=\displaystyle\int_{x_0}^x\frac{d\xi}{P(\xi)},
\end{equation}
where $x_0,\; x \in (a,b),$   transforms Eq. \eqref{mainlin1} into the following second order differential equation:
\[
y^{\prime\prime}(t)+\alpha y^\prime(t)+\beta y(t)=0. 
\]
This differential equation is with constant coefficients which can be solved by using the elementary techniques of solving second order differential equations. For illustration, we present the following examples:

\begin{example}
Consider the well-known Cauchy-Euler's Equation
\[
x^2y^{\prime\prime}(x)+(\alpha +1) xy^\prime(x)+\beta y=0,\;x>0.
\]
Then $P(x)=x$, and the $t-$transformation is $t=\ln(x)$, which transforms the equation into
\[
y^{\prime\prime}(x)+\alpha y^\prime(x)+\beta y=0
\]
\end{example}

\begin{example}
Consider the Chebyshev's Equation 
\[
\left[1-x^2\right]y^{\prime\prime}(x)-2xy^\prime(x)+n^2 y=0,\;|x|<1.
\]
Then $P(x)=\sqrt{1-x^2}$, and the $t-$transformation is $t=\sin^{-1}(x)$, which transforms the equation into
\[
y^{\prime\prime}(x)+n^2 y=0.
\]
Using this transformation, the solution of Chebyshev's Equation is given by
\[
y(x)=A\cos(n\sin^{-1}(x))+B\sin(n\sin^{-1}(x)).
\]
\end{example}

\begin{example}
Consider the  Hypergeometric Equation 
\begin{equation}\label{hypereq}
x(1-x)y^{\prime\prime}(x)+\frac{1}{2}(1-2x)y^\prime(x)+a^2 y=0,\;x\in(0,1).
\end{equation}
Then $P(x)=\sqrt{x(1-x)}$, and the $t-$transformation is $t=\sin^{-1}(2x-1).$ This  transforms the  equation into
\[
y^{\prime\prime}(t)+a^2 y(t)=0.
\]
%Using this transformation, the solution of Hypergeometric Equation is given by
Hence, the solution of  \eqref{hypereq} is given by
\[
y(x)=A\cos(a\sin^{-1}(2x-1))+B\sin(a\sin^{-1}(2x-1)).
\]
\end{example}
For certain functions, $h(x)\in C(a,b)$, for some open interval $(a,b)\in \R$ , we can solve the nonhomogeneous second order differential equation 
\begin{equation}\label{main33}
\left[P(x)\right]^2y^{\prime\prime}(x)+P(x)\left[\alpha+ P^\prime(x) \right]y^\prime(x)+\beta y=h(x).  
\end{equation}
Particularly, when $h(x)$ can be written in the form $H(t),$ where $t=\int_{x_0}^x\frac{d\xi}{P(\xi)}.$ The following example shows this idea:
\begin{example}
Consider the  nonhomogeneous differential equation  
\begin{equation}\label{linexample}
x(1-x)y^{\prime\prime}(x)+\frac{1}{2}(1-2x)y^\prime(x)+a^2 y=2x,\;x\in(0,1).
\end{equation}
Then $P(x)=\sqrt{x(1-x)}$. The $t-$transformation is $t=\sin^{-1}(2x-1)$. This transforms the equation into
\[
y^{\prime\prime}(t)+a^2 y(t)=1+\sin(t).
\]
Hence, the solution of equation \eqref{linexample} is given by
\begin{equation*}
y(x)=\left\{
\begin{array}{ll}
A\cos(a\sin^{-1}(2x-1))+B\sin(a\sin^{-1}(2x-1))+\displaystyle \frac{2x-1}{a^2-1}+\displaystyle\frac{1}{a^2},\;\text{if}\;\;a\neq \pm 1,\\
A\cos(\sin^{-1}(2x-1))+B(2x-1)+\displaystyle \frac{1}{2}(1-2x)\sin^{-1}(2x-1)+1,\;\text{if}\;\;a= \pm 1.
\end{array}
\right.
\end{equation*}
\end{example}

%=============================================================================
\section{Solving $f-$type Second Order Differential Equations}
In this section, we  solve the following class of second order  nonlinear differential equation
\begin{equation}\label{part2eq1}
a_2\left(f^\prime(y)y^{\prime\prime}+(y^\prime)^2f^{\prime\prime}(y)\right)+a_1f^\prime(y)y^\prime+a_0f(y)=g(x),
\end{equation}
where $a_2,a_1$ and $a_0$ are constants, and  $f\in C^2(a,b)$, for some  open interval $(a,b) \subset \R$. In this section,  we also solve the following class of second order  nonlinear differential equation:
\begin{equation}\label{part2eq2}
p(x)\left(f^\prime(y)y^{\prime\prime}+(y^\prime)^2f^{\prime\prime}(y)\right)+\frac{1}{2}p^\prime(x)f^\prime(y)y^\prime+a_0f(y)=0,
\end{equation}
where $p(x)$ is a positive and differentiable function on an open interval $(a,b) \subset \R$, and  $f\in C^2(c,d)$, for some open interval $(c,d)\subset \R$. To solve \eqref{part2eq1},  let $z=f(y)$. Hence, $z^\prime=f^\prime(y)y^\prime$, and $z^{\prime\prime}=f^\prime(y)y^{\prime\prime}+(y^\prime)^2f^{\prime\prime}(y).$  Substitute $z$, $z^\prime$ and $z^{\prime\prime}$ in equation \eqref{part2eq1}, we get% can be transformed into the following nonhomogeneous differential equation:
\begin{equation}\label{part2eqcc}
a_2z^{\prime\prime}+a_1z^\prime+a_0z=g(x), 
\end{equation}
%which can be solved by elementary techniques.
Similarly,  equation \eqref{part2eq2}  becomes %transformed 
 %into the following linear second order differential equation:
\begin{equation}\label{part2eq2l}
p(x)z^{\prime\prime}+\frac{1}{2}p^\prime(x)z^\prime+a_0z=0,
\end{equation}
which is the linear form of \eqref{main}. Therefore, it can be solved by using the technique described in Example \ref{ex3}.  To illustrate the procedure of solving \eqref{part2eq1} and \eqref{part2eq2}, we present the following examples:% to explain solve the differential equations 
\begin{example}
Consider Langumir Equation, with a slightly modification, 
\begin{equation}\label{part2ex1}
3yy^{\prime\prime}+3(y^\prime)^2+4yy^\prime+y^2=1.
\end{equation}
The original Langumir Equation is given by
\begin{equation*}
3yy^{\prime\prime}+(y^\prime)^2+4yy^\prime+y^2=1
\end{equation*}
which originally appears in connection with the theory of current flow from hot cathode to an anode in a hight vacuum \cite{Ames,Langmuir}. To solve \eqref{part2ex1}, we let $z=\displaystyle\frac{y^2}{2}$. Then $z^\prime=yy^\prime$ and $z^{\prime\prime}=yy^{\prime\prime}+(y^\prime)^2.$ Hence, equation \eqref{part2ex1} becomes
\[
3z^{\prime\prime}+4z^\prime+2z=1.
\]
The solution of this equation is $$z(x)=e^{-\frac{2}{3}x}\left(A\cos\left(\frac{\sqrt 2}{3}x\right)+B\sin\left(\frac{\sqrt 2}{3}x\right)\right)+\frac{1}{2}.$$ Hence, the solution of \eqref{part2ex1} is given by
\[
y^2=2e^{-\frac{2}{3}x}\left(A\cos\left(\frac{\sqrt 2}{3}x\right)+B\sin\left(\frac{\sqrt 2}{3}x\right)\right)+1.
\]
\end{example}
\begin{example} 
Consider the initial value problem
\[
\left\{
\begin{array}{ll}
y^{\prime\prime}+(y^\prime)^2+1=(\cos \omega x)e^{-y},\; \;\omega\neq \pm 1,\\
y(0)=y^\prime(0)=0.
\end{array}
\right.
\]
This problem is equivalent to 
\[
\left\{
\begin{array}{ll}
\left(y^{\prime\prime}+(y^\prime)^2\right)e^y+e^y=(\cos \omega x),\; \;\omega\neq \pm 1,\\
y(0)=y^\prime(0)=0.
\end{array}
\right.
\]
%which is of the form that is given in Equation \eqref{part2eq1}. By letting
Let $z=e^y.$ Then $z^\prime=y^\prime e^y$ and $z^{\prime\prime}=y^{\prime\prime}e^y+(y^\prime)^2e^y$. By substituting $z$, $z^\prime$ and $z^{\prime\prime}$ in the above initial value problem, we get% we can rewrite the above initial value problem as 
\[
\left\{
\begin{array}{ll}
z^{\prime\prime}+z=\cos \omega x,\;\; \omega \neq \pm 1,\\
z(0)=1,\;z^\prime(0)=0.
\end{array}
\right.
\]
The solution of this problem is $z(x)=\frac{1}{1-\omega^2}\left( \cos \omega x-\omega^2\cos x\right),\;\omega\neq \pm1.$ Therefore,
$y(x)=\ln \left(\frac{1}{1-\omega^2}\left(\cos \omega x-\omega^2\cos x \right)\right),\; \omega \neq \pm 1.$
\end{example}

\begin{example} Let $\phi(x)$ be a positive and differentiable function on an open interval $(a,b) \subset \R$, and consider the differential equation
\[
\phi(x)\left(y^{\prime\prime}+(y^\prime)^2\right)+\frac{1}{2}\phi^\prime(x) y^\prime +\lambda =0.
\]
By multiplying this equation by $e^y$, we get
\[
\phi(x)\left(y^{\prime\prime}+(y^\prime)^2\right)e^y+\frac{1}{2}\phi^\prime(x)y^\prime e^y+\lambda e^y=0.
\]
%This equation is of the form given in Equation \eqref{part2eq2}. Therefore, if we
Let $z=e^y.$ Then $z^\prime=y^\prime e^y$ and $z^{\prime\prime}=y^{\prime\prime}e^y+(y^\prime)^2e^y$. By substituting $z$, $z^\prime$ and $z^{\prime\prime}$ in the above differential equation, we get %the following linear differential equation:
\[
\phi(x)z^{\prime\prime}+\frac{1}{2}\phi^\prime(x)z^\prime+\lambda z=0.
\]
The solution of this equation is (see Example \ref{ex3}) 
\[
z(x)=C_1\sin\left(\lambda\int_{x_0}^x \frac{d\xi}{\sqrt{\phi(\xi)}}d\xi \right)+C_2\cos\left(\lambda\int_{x_0}^x \frac{d\xi}{\sqrt{\phi(\xi)}}d\xi \right).
\]
Therefore,
\[
y(x)=\ln\left[C_1\sin\left(\lambda\int_{x_0}^x \frac{d\xi}{\sqrt{\phi(\xi)}}d\xi \right)+C_2\cos\left(\lambda\int_{x_0}^x \frac{d\xi}{\sqrt{\phi(\xi)}}d\xi \right)\right].
\]
\end{example}
%=================================================================================
\section{Solving $f$-type Second Order Differential Equations that can be  Transformed into Exact Second Order Differential Equations}
In this section, we solve the following class of  second order nonlinear differential equations:
\begin{equation}\label{part3eq1}
a_2(x,y,y^\prime)\left(f^\prime(y)y^{\prime\prime}+f^{\prime\prime}(y)(y^\prime)^2\right)+a_1(x,y,y^\prime)(f^\prime(y)y^\prime)+a_0(x,y,y^\prime)=0,
\end{equation}
where $f(y)$ is an invertible function and  $f\in C^2(a,b)$. To solve this class of differential equations, we let $z=f(y)$. Then $z^\prime=f^\prime(y)y^\prime$ and $z^{\prime\prime}=f^{\prime\prime}(y)(y^\prime)^2+f^\prime(y)y^{\prime\prime}.$ Moreover, we let $y=f^{-1}(z).$ Then $y^\prime=\frac{z^\prime}{f^{\prime}\left(f^{-1}(z)\right)}.$ Hence, equation \eqref{part3eq1} can be transformed into the following differential equation:
\begin{equation}\label{part3eq2}
a_2\left(x,f^{-1}(z),\frac{z^\prime}{f^\prime\left(f^{-1}(z)\right)}\right)z^{\prime\prime}+a_1\left(x,f^{-1}(z),\frac{z^\prime}{f^\prime\left(f^{-1}(z)\right)}\right)z^\prime+a_0\left(x,f^{-1}(z),\frac{z^\prime}{f^\prime\left(f^{-1}(z)\right)}\right)=0
\end{equation}
Assume that \eqref{part3eq2} is exact, then it can be solved. To explain the procedure of solving such differential equations, we consider the following example:
%--------------------------------------------------------------------------------
\begin{example}
Consider the  second order nonlinear differential equation
\begin{equation}\label{example6}
\left\{
\begin{array}{ll}
e^y\left[y^{\prime\prime}+(y^\prime)^2\right]+12xe^{4y}y^\prime+\left(3e^{4y}-1\right)=0,\\
y(0)=\ln2,\;y^\prime(0)=0.
\end{array}
\right.
\end{equation}
Let $z=e^y$. Then $z^\prime=e^{y}y^\prime$ and $z^{\prime\prime}=e^{y}y^{\prime\prime}+e^y(y^\prime)^2.$ Hence, Eq. \eqref{example6} becomes
\begin{equation}\label{ex6exact}
\left\{
\begin{array}{ll}
z^{\prime\prime}+12xz^3z^\prime+\left(3z^4-1\right)=0,\\
z(0)=2,\;z^\prime(0)=0.
\end{array}
\right.
\end{equation}
Therefore, $a_2(x,z,z^\prime)=1,$ $a_1(x,z,z^\prime)=12xz^3,$ and $a_0(x,z,z^\prime)=\left(3z^4-1\right).$ In addition, we have %This implies that,
\begin{equation}\label{exactcond22}
\frac{\partial a_2}{\partial z}=\frac{\partial a_1}{\partial z^\prime}=0, \;\; \frac{\partial a_2}{\partial x}=\frac{\partial a_0}{\partial z^\prime}=0,\; \text{and}\; \;\frac{\partial a_1}{\partial x}=\frac{\partial a_0}{\partial z}=12z^3,
\end{equation}
Therefore, equation \eqref{ex6exact} is exact differential equation. Hence, its first integral exists and it is given by % can be reduced to the following first order differential equation
%\begin{eqnarray*}
%c&=&\int_{0}^{x}a_0(\alpha,z,z^\prime)d\alpha+\int_{2}^{z}a_1(0,\beta,z^\prime)d\beta+\int_{0}^{z^\prime}a_2(0,2,\gamma)d\gamma,\\
%&=&\int_{0}^{x}(3z^4-1)d\alpha+\int_{0}^{z^\prime}d\gamma,\\
%&=& z^\prime+(3z^4-1)x.\\
%\end{eqnarray*}
%By applying the initial data, we get $c=0$. Hence, we get
\[
z^\prime+3xz^4-x=0.
\]
For which an implicit solution of this equation can be obtained by separating the variables, and so, $y(x)=\ln(z(x)).$ 
\end{example}
\begin{remark}
Assume that  \eqref{part3eq2} is not exact. Then an integrating factor of \eqref{part3eq2} could be exist. Hence, it can be transformed into an exact differential equation (see \cite{AlAhmad}). %In this case, Equation \eqref{part3eq2} becomes exact . We explain this in the following example:
To explain the procedure of solving \eqref{part3eq2} in case it is not exact, we present the following example: 
\end{remark}
\begin{example} Consider the second order  nonlinear differential equation
\begin{equation}\label{part3ex21}
xe^y\left(2x+e^y\right)\left(y^{\prime\prime}+(y^\prime)^2\right)+x\left(x+e^y\right)y^\prime+\left(3x+e^y\right)=0.
\end{equation}
By multiplying this equation by $e^y$, we get
\begin{equation}\label{part3ex22}
xe^{2y}\left(2x+e^y\right)\left(y^{\prime\prime}+(y^\prime)^2\right)+x\left(x+e^y\right)e^yy^\prime+e^y\left(3x+e^y\right)=0.
\end{equation}
Let $z=e^y$. Then $z^\prime=e^{y}y^\prime$ and $z^{\prime\prime}=e^{y}y^{\prime\prime}+e^y(y^\prime)^2.$ Hence, by substituting $z$, $z^\prime$ and $z^{\prime\prime}$ in \eqref{part3ex22}, we get 
\begin{equation}\label{part3ex23}
xz(2x+z)z^{\prime\prime}+x(x+z)z^\prime+z(3x+z)=0.
\end{equation}
This equation is not exact since $\frac{\partial a_2}{\partial z}=2(x+z)\neq0=\frac{\partial a_1}{\partial z^\prime}$. An integrating factor of this second order nonlinear differential equation exists, and it is given by $\mu(x,z)=\displaystyle \frac{1}{xz(2x+z)}.$ Multiplying \eqref{part3ex23} by $\mu(x,z)$, we get
\begin{equation}\label{part3ex24}
z^{\prime\prime}+\frac{(x+z)}{z(2x+z)}z^\prime+\frac{(3x+z)}{x(2x+z)}=0.
\end{equation}
Clearly,
\begin{equation}
\frac{\partial a_2}{\partial z}=\frac{\partial a_1}{\partial z^\prime}=0, \;\; \frac{\partial a_2}{\partial x}=\frac{\partial a_0}{\partial z^\prime}=0,\; \text{and}\; \;\frac{\partial a_1}{\partial x}=\frac{\partial a_0}{\partial z}=\frac{-1}{(2x+z)^2}\,.
\end{equation}
Therefore, the differential equation \eqref{part3ex24} is exact, and its first integral is given by % Hence, it can be reduced to the following first order differential equation:
\begin{equation}
c= z^\prime+\ln\left(xz\sqrt{2x+z}\right).
\end{equation}
This first order differential equation can be solved by using the elementary techniques of solving first order differential equations. Hence, $y(x)=\ln(z(x))$.
\end{example}
Finally, we consider the nonhomogeneous second order linear differential equation
\[
a_2(x)y^{\prime\prime}+a_1(x)y^\prime+a_0(x)y=h(x),
\]
where $a_2(x)\neq 0$, $a_1(x)$, and $a_0(x)$ are differentiable functions on an open interval $(a,b) \subset \R$. This equation admits an integrating factor $\mu(x)=\displaystyle\frac{1}{a_2(x)}$ provided that  $W(a_2,a_1)(x)=a_0(x)a_2(x),$ where $W(a_2,a_1)(x)=a_2(x)a^\prime_1(x)-a_1(x)a^\prime_2(x)$. For this case, we present the following example:
\begin{example}
consider the second order linear differential equation
\[
e^x y^{\prime\prime}+\cos x y^\prime -(\cos x+\sin x)y=h(x).  
\]
By multiplying this equation by the integrating factor $e^{-x}$, we get
\[
 y^{\prime\prime}+e^{-x}\cos x y^\prime -e^{-x}(\cos x+\sin x)y=h(x)e^{-x}. 
\]
This equation can be written as
\[
 \frac{d}{dx}\left[y^\prime+(e^{-x}\cos x) y\right]=h(x)e^{-x}  
\]
Hence, its first integral is given by % can be reduced into the following first order differential equation:

\[
 y^\prime+(e^{-x}\cos x) y=\displaystyle\int^x h(\xi)e^{-\xi}d\xi+c_1 
\]
which can be solved by using the elementary techniques of solving first order differential equations.
\end{example}
%=================================================================================
\section{Concluding Remarks}
In this paper, we  solved some classes of second order differential equation. In fact, we solved the following classes of second order differential equations:
\begin{enumerate}
\item The Chebyshev's type of second order differential equation \begin{equation}\label{CRmain}
p(x)y^{\prime\prime}(x)+\frac{1}{2}p^\prime(x)y^\prime(x)+f(\sqrt{p(x)}\;y^\prime(x),y(x))=0,\;\;\; x\in (a,b),
\end{equation}
where $p(x)$ is a positive and differentiable function on an open interval $ (a,b)\subset \R$, and  $f(\sqrt{p(x)}\;y^\prime(x),y(x))$ is a continuous function on some domain $D\subset \R^2$. 
\item The $f-$type of second order differential equations
\begin{itemize}
\item [a)] \begin{equation}\label{CRpart2eq}
a_2\left(f^\prime(y)y^{\prime\prime}+(y^\prime)^2f^{\prime\prime}(y)\right)+a_1f^\prime(y)y^\prime+a_0f(y)=g(x),
\end{equation}
where $a_2,a_1$ and $a_0$ are constants, and the function $f(y)$ is of  $C^2-$class on some open interval $ (a,b) \subset \R$, and 
\item [b)] 
\begin{equation}\label{CRpart2eeq}
p(x)\left(f^\prime(y)y^{\prime\prime}+(y^\prime)^2f^{\prime\prime}(y)\right)+\frac{1}{2}p^\prime(x)f^\prime(y)y^\prime+a_0f(y)=0,
\end{equation}
where $p(x)$ is a positive and differentiable function on some open interval $(a,b) \subset \R$, and  $f\in C^2(c,d)$,  for some  open interval $(c,d) \subset \R$.
\end{itemize}
\item $f$-type second order differential equations that can be transformed into  exact second order Differential Equations
\begin{equation}
a_2(x,y,y^\prime)\left(f^\prime(y)y^{\prime\prime}+f^{\prime\prime}(y)(y^\prime)^2\right)+a_1(x,y,y^\prime)(f^\prime(y)y^\prime)+a_0(x,y,y^\prime)=0,
\end{equation}
where the function $f(y)$ is an invertible function and  $f\in C^2(a,b)$, for some  open interval $(a,b) \subset \R$. 
\end{enumerate}
Moreover, we presented some examples to explain our approach of solving the above classes of second order differential equation.  
%=================================================================================

\end{document}